\documentclass[]{article}

\usepackage{amssymb}
\usepackage{amsmath}
\usepackage{algpseudocode}
\usepackage{graphicx}
\usepackage{url,here}
\usepackage[backref,colorlinks=true]{hyperref}
\usepackage{multirow}
\usepackage{algorithm}
\allowdisplaybreaks

\begin{document}
\newenvironment {proof}{{\noindent\bf Proof.}}{\hfill $\Box$ \medskip}

\newtheorem{theorem}{Theorem}[section]
\newtheorem{lemma}[theorem]{Lemma}
\newtheorem{condition}[theorem]{Condition}
\newtheorem{proposition}[theorem]{Proposition}
\newtheorem{remark}[theorem]{Remark}
\newtheorem{definition}[theorem]{Definition}
\newtheorem{hypothesis}[theorem]{Hypothesis}
\newtheorem{corollary}[theorem]{Corollary}
\newtheorem{example}[theorem]{Example}
\newtheorem{descript}[theorem]{Description}
\newtheorem{assumption}[theorem]{Assumption}

\newcommand{\ba}{\begin{align}}
\newcommand{\ea}{\end{align}}

\def\P{\mathbb{P}}
\def\R{\mathbb{R}}
\def\E{\mathbb{E}}
\def\N{\mathbb{N}}
\def\Z{\mathbb{Z}}

\renewcommand {\theequation}{\arabic{section}.\arabic{equation}}
\def \non{{\nonumber}}
\def \hat{\widehat}
\def \tilde{\widetilde}
\def \bar{\overline}

\def\ind{{\mathchoice {\rm 1\mskip-4mu l} {\rm 1\mskip-4mu l}
{\rm 1\mskip-4.5mu l} {\rm 1\mskip-5mu l}}}

\title{\Large\ {\bf Unbiased estimation of parameter sensitivities for stochastic chemical reaction networks}}

\author{Ankit Gupta and Mustafa Khammash \\
Department of Biosystems Science and Engineering \\ ETH Zurich \\  Mattenstrasse 26 \\ 4058 Basel, Switzerland. 
}
\date{\today}
\maketitle
\begin{abstract}
Estimation of parameter sensitivities for stochastic chemical reaction networks is an important and challenging problem. Sensitivity values are important in the analysis, modeling and design of chemical networks. They help in understanding the robustness properties of the system and also in identifying the key reactions for a given outcome. In a discrete setting, most of the methods that exist in the literature for the estimation of parameter sensitivities rely on Monte Carlo simulations along with finite difference computations. However these methods introduce a bias in the sensitivity estimate and in most cases the size or direction of the bias remains unknown, potentially damaging the accuracy of the analysis.
In this paper, we use the random time change representation of Kurtz to derive an exact formula for parameter sensitivity. This formula allows us to construct an unbiased estimator for parameter sensitivity, which can be efficiently evaluated using a suitably devised Monte Carlo scheme. The existing literature contains only one method to produce such an unbiased estimator. This method was proposed by Plyasunov and Arkin and it is based on the Girsanov measure transformation. By taking a couple of examples we compare our method to this existing method. Our results indicate that our method can be much faster than the existing method while computing sensitivity with respect to a reaction rate constant which is small in magnitude. This rate constant could correspond to a reaction which is \emph{slow} in the reference time-scale of the system.
Since many biological systems have such slow reactions, our method can be a useful tool for sensitivity analysis.   
\end{abstract}

\medskip

\setcounter{equation}{0}

\section{Introduction}
Stochastic models for chemical reactions networks have become increasingly popular in the past few years. Their appeal comes from the fact that molecules in a chemical system always display randomness in their dynamics. This randomness cannot be ignored when the molecules are present in low numbers, as it can have a significant effect on the overall properties of the dynamics. Therefore one needs to consider stochastic models to account for this randomness. Recently many problems in Biology \cite{Noise,ArkinAdams,Elowitz} and Chemistry \cite{Haseltine} have been studied using such models. 
Typically, a chemical reaction network depends on various kinetic parameters whose values may be uncertain or difficult to measure with high precision. In such cases, one would like to determine how sensitive a given output of the system is to small changes in the parameter values. If an outcome is highly sensitive, then greater time and effort may be invested in determining that parameter accurately. Sensitivity analysis can also be used in fine-tuning a certain output (see \cite{Feng}) or understanding the robustness properties of a biological system (see \cite{Stelling}). 

Consider a system consisting of $d$ chemical species. We assume that the system is well-stirred and hence its state at any time can be described by a vector in 
$\mathbb{N}_{0}^d$ whose $i$-th component is the non-negative integer corresponding to the number of molecules of the $i$-th species.  These chemical species interact through $K$ predefined  reaction channels and every time the $j$-th reaction fires, the state of the system is displaced by the $d$-dimensional \emph{stoichiometric vector} $\zeta_j \in \Z^d$.  Given that the state of the system is $x$, the rate at which the $j$-th reaction fires is given by the \emph{propensity function} $\lambda_j(x)$.
In a stochastic setting, such a chemical reaction network can be modeled by a continuous time Markov process over $\mathbb{N}^d_0$. See \cite{DASurvey} for a survey of such models. It can be shown that the probability mass function of this Markov chain evolves according to the chemical master equation (CME). Solving the CME analytically or even numerically is usually quite difficult except for very simple systems (see \cite{FSP}). However it is easy to generate the sample paths of this process using Monte Carlo simulations (see \cite{GP,NR}).

Now suppose that the propensity functions of the reaction network depend on a scalar parameter $\theta$. We denote these propensity functions by $\lambda_j(x,\theta)$ for $j =1,\dots,K$ and the resulting Markov process by $
\{ X_\theta(t) : t \geq 0 \}$. Given a function $f : \N_0^d \to \R$ and an observation time $T \geq 0$, our goal is to compute the quantity
\begin{align}
\label{sensitivity1}
S_{\theta}(f,T) = \frac{\partial}{  \partial \theta} \E \left( f(X_{\theta}(T)) \right).
\end{align}• 
Here $f(X_{\theta}(T))$ is the output of interest and $S_{\theta}(f,T)$ evaluates how the expected value of this output changes with infinitesimal changes in the parameter $\theta$. Since simulating the paths of this Markov process is easy, many methods use a finite difference scheme such as 
\begin{align}
\label{fdmscheme}
S_{\theta,h }(f,T) = \frac{1}{h}\E \left(  f(X_{\theta + h}(T)) - f(X_{\theta}(T))  \right),
\end{align}
for a small $h$ to estimate $S_{\theta}(f,T)$. The processes $X_\theta$ and $X_{\theta+h}$ can either be simulated independently (see \cite{IRN}) or they can be coupled in an intelligent way to reduce the variance of the associated estimator (see \cite{KSR1,DA}). Finite difference schemes are easy to implement but they have a few drawbacks. They introduce a bias in the sensitivity value and one cannot recover the exact value $S_\theta(f,T)$, even by taking infinitely many samples. In most cases, the magnitude and sign of this bias is unknown which may cause problems in certain applications. One can reduce the size of the bias by picking a small $h$. However for small values of $h$, the variance of the sensitivity estimator can blow up (see the discussion in \cite{KSR1}), making it necessary to generate extremely large number of samples to obtain the sensitivity estimate within a tight confidence interval. Hence there is a clear trade-off between the bias and the computational cost which implies that such schemes are efficient as long as one is willing to tolerate some amount of (unknown) bias. Another source of difficulty with finite difference schemes is that to estimate the sensitivities with respect to several parameters, one needs to estimate it with respect to each parameter separately. This can be quite cumbersome for large networks. One can avoid this problem by using another approach which estimates the sensitivity value using pathwise differentiation (see \cite{KSR2}). In this method, the problem of estimating the sensitivity is \emph{regularized} in the sense that  
\[ \frac{\partial}{  \partial \theta}  \E  \left(  \int_{T -w}^{T+w} \frac{1}{2w} f(X_{\theta}(t)) d t \right),\]
is computed as a proxy for $S_{\theta}(f,T)$. Here $w$ is the half-width of the \emph{regularizing window}. With this scheme, the sensitivities with respect to multiple parameters can be estimated simultaneously. However it also produces a biased estimate and its performance depends crucially on the parameter $w$ which has to be determined carefully to achieve the desired level of accuracy.

If one is interested in finding an unbiased estimate for the parameter sensitivity, then for this purpose there is only one approach in the existing literature. This approach was proposed by Plyasunov and Arkin in \cite{Gir} and it relies on the Girsanov measure transformation. Henceforth we shall call this method as the Girsanov method. It has the advantage of being computationally easy to implement and unlike the other methods mentioned above, one does not need to tune other parameters (like $h$ or $w$) to achieve the required statistical accuracy. Moreover the sensitivities with respect to several parameters can be estimated together in a single run. Since the method is unbiased we can get better and better estimates simply by taking larger and larger samples. However in many situations the estimator produced by this method has a large variance. One such situation that commonly arises is when the sensitive parameter $\theta$ is a reaction rate constant which is \emph{small} in size (see Example \ref{example1}).  
In this case, the variance of the sensitivity estimator is large, making it difficult to estimate the sensitivity efficiently. We shall discuss this issue in greater detail in the next section. 

  
The above discussion shows that it would be desirable to have a \emph{new} method, which provides an unbiased estimate for the parameter sensitivity and also performs better than the Girsanov method in certain situations. The goal of this paper is to present such a method.  
Using the random time change representation of Kurtz (see Chapter 7 in \cite{EK}) we first derive an explicit expression for $S_{\theta}(f,T)$ (defined by \eqref{sensitivity1}). This expression is derived by exploiting the coupling that was introduced in \cite{DA} to construct an efficient estimator for $S_{\theta,h }(f,T) $ (see \eqref{fdmscheme}).
Based on this expression for $S_{\theta}(f,T)$ we construct an unbiased estimator for the parameter sensitivity. 
Our method has two main advantages over the Girsanov method. It requires much fewer number of samples to produce the desired estimate and its performance does not deteriorate as the magnitude of the sensitive parameter gets smaller.
In fact, our method can also be used to compute the sensitivity with respect to a rate constant that is set to $0$. This can tell us how sensitive a given output is to the presence or absence of a certain reaction channel and this information can help us weed out the redundant reactions in a reaction network. Such a computation is not possible with the Girsanov method. Our method also has its disadvantages in comparison to the Girsanov method. It is harder to implement, requires more memory to run and the computational cost of generating each sample from the required distribution is high. Our results indicate that our method can be considerably more efficient than the Girsanov method, while computing the sensitivity with respect to a reaction rate constant which is \emph{small} in size. This would correspond to the reaction which is \emph{slow} in the reference time-scale of the system (determined by the observation time period $[0,T]$). As mentioned in the preceding paragraph, this is one of the situations where the Girsanov method performs poorly. Since many biochemical networks have wide variations in the rate constants of the constituent reactions, our method could be a useful tool for sensitivity analysis.

In our opinion, the main contribution of this paper is to present a new formula for the parameter sensitivity. This formula expresses the sensitivity as an expectation of a certain random variable which depends on the path of the underlying Markov process. However we shall see in the next section, that evaluating this random variable requires us to compute several expectations of functions of our Markov process at various times and various initial states. In principle we can compute such expectations as and when they are needed by simulating many \emph{new} paths of the Markov process, but this approach can render the problem computationally intractable for most systems of interest. To circumvent this issue we devise a Monte Carlo scheme in which all the required expectations are estimated using a fixed number of \emph{auxiliary} paths. This approach is outlined in Section 3. A more detailed description along with full implementation details is given in \cite{report}.
Finally we would like to mention that our sensitivity formula can be used for any continuous time Markov process over a discrete lattice. Other than stochastic reaction networks, such processes arise naturally in queuing theory and population modeling.

This paper is organized as follows. In Section 2 we present our main result which gives a formula for the parameter sensitivity. Through a couple of examples we motivate why the estimator based on our formula can perform better than the Girsanov method in certain situations.
In Section 3, we formally describe our method for estimating the parameter sensitivity and compare its performance with the Girsanov method. We also remark how our method fits into the grand scheme of sensitivity analysis for stochastic reaction networks. In Section 4 we prove the result mentioned in Section 2. Finally, in Section 5 we conclude and present directions for future research.

\section{The Main Result}

Recall the description of the chemical reaction network from the previous section. We suppose that the propensity functions depend on a scalar parameter $\theta$ and hence denote these functions 
by $\lambda_j(x,\theta)$ for $j =1,\dots,K$. In a stochastic setting, the dynamics can be modeled by a continuous time Markov process whose generator\footnote{The generator of a Markov process is an operator which specifies the rate of change of the distribution of the process. See Chapter 4 in \cite{EK} for more details.} is given by
 \begin{align}
\label{genctmc2}
\mathbb{A}_{\theta}f(x) = \sum_{k = 1}^K \lambda_k(x,\theta) \Delta_{\zeta_k} f(x),
\end{align} 
where $f$ is a bounded function from $\N^d_0$ to $\R$. Let $\{ X_\theta(t) : t \geq 0\}$ be a Markov process with this generator. The random time change representation of Kurtz (see Chapter 7 in \cite{EK}) allows us to write 
\begin{align}
\label{tcrep2}
X_{\theta}(t) = X_{\theta}(0) + \sum_{k = 1}^{K} Y_k \left( \int_{0}^{t} \lambda_k(X_{\theta}(s), \theta) ds\right) \zeta_k,
\end{align}
where $\{Y_k : k = 1,\dots,K \}$ is a family of independent unit rate Poisson processes. We define the function $\lambda_0$ as the sum of propensities
\begin{align}
\label{defnlambda0}
\lambda_0(x,\theta) = \sum_{i = 1}^K \lambda_i(x,\theta).
\end{align}
Throughout the paper $\| \cdot\|$ will refer to the standard $1$-norm on $\R^n$ given by $\| x\| = \sum_{i=1}^{n}|x_i|$. Moreover the vector $\bar{1}$ will denote the vector of all ones in $\R^n$. 
We now state a condition on real-valued functions on our state space. 
\begin{condition}
\label{cond:bddness} A function $f : \mathbb{N}^d_0 \to \mathbb{R}$ satisfies this condition if there exist constants $C ,r >0$ such that 
\begin{align*}
|f(x)| \leq C(1 + \|x\|^r) \textrm{ for all } x \in \N_0^d.
\end{align*}
\end{condition}
We now state a condition on the propensity functions.
\begin{condition}
\label{cond:propensityfunctions}
We say that the propensity functions $\lambda_1,\dots,\lambda_K$ satisfy this condition at the parameter value $\theta$, if the following is true.
\begin{itemize}
\item[(A)] For any fixed $x \in \N_0^d$, each $\lambda_k(x,\cdot)$ is twice-continuously differentiable in a neighborhood of $\theta$.
\item[(B)] For each $k$, the functions $\lambda_k( \cdot,\theta)$ and $\partial \lambda_k (\cdot,\theta)/ \partial \theta$ satisfy Condition \ref{cond:bddness}. Moreover there exists an $\epsilon > 0$ such that
$\sup_{ \xi\in (\theta - \epsilon , \theta + \epsilon)} |\partial^2 \lambda_k (\cdot,\xi)/ \partial \theta^2|$ also satisfies Condition \ref{cond:bddness}.
\item[(C)] For any $k$ and $x \in \N^d_0$, if $\lambda_k(x,\theta) > 0$ then the vector $(x + \zeta_k)$ has all non-negative components. 
\item[(D)] Let $P$ be the set of indices of those reactions which have a net positive affect on the total population. That is
\begin{align}
\label{setofpositivereactions}
P = \left\{ k = 1,2,\dots,K : \langle \bar{1}, \zeta_k \rangle > 0\right\}.
\end{align}
Then there exists a $C_\lambda > 0$ such that for all $x \in \mathbb{N}^d_0$ we have 
\begin{align*}
\sum_{ k = 1, k \in P }^K \lambda_k(x,\theta)  \leq C_\lambda (1 + \| x\|). 
\end{align*}
\end{itemize}
\end{condition}
Parts (A) and (B) are technical requirements for our main result. Part (C) prevents the Markov process from leaving the state space $\N^d_0$.  
Part (D) is needed to ensure that all the moments of the Markov process can be bounded uniformly in time (see Lemma \ref{lemma1}). Informally, this condition says that all the reactions that add molecules into the system have orders $0$ or $1$. If there is a compact set $S$ such that for each $k$, $\lambda_k(x,\theta) = 0$ for all $x \notin S$, then part (D) is trivially satisfied.

If $\{ X_{\theta}(t) : t \geq 0 \}$ is a Markov process with generator $\mathbb{A}_{\theta}$ then let
\begin{align}
\label{defpsixtheta}
\Psi_{\theta}(x,f,t) = \mathbb{E} \left(  f(X_\theta(t)) \middle\vert X_{\theta}(0) = x\right).
\end{align}
Also for each $k =1 ,\dots, K$ define 
\begin{align}
\label{defrktheta}
R_{\theta}(x,f,t,k) =  \int_{0}^{t} \left(  \Psi_{\theta}(x + \zeta_k,f,s)  -   \Psi_{\theta}(x,f,s)   - \Delta_{\zeta_k} f(x) \right) e^{ - \lambda_0(x,\theta) (t-s)} ds.
\end{align}
We are now ready to state our main result. It gives an explicit expression for $S_\theta(f,T)$ which is defined by \eqref{sensitivity1}.
\begin{theorem}
\label{thm:sens}
Pick a $x_0 \in \N^d_0$ and a function $f :  \N^d_0 \to \R$ satisfying Condition \ref{cond:bddness}. Assume that the propensity functions $\lambda_1,\dots,\lambda_K$ satisfy Condition \ref{cond:propensityfunctions} at $\theta$. Let $\{ X_\theta(t) : t \geq 0\}$ be the $\N^d_0$-valued Markov process with generator $\mathbb{A}_\theta$ starting at $x_0$ and let 
$\sigma_i$ be its $i$-th jump time\footnote{We define $\sigma_0 = 0$ for convenience.} for $i =0,1,2,\dots$.  
Then the sensitivity value $S_{\theta}(f,T)$ is well-defined and is equal to
\begin{align*}
S_{\theta}(f,T)    = \frac{\partial}{\partial \theta} \Psi_\theta( x_0,f,T) = \E \left( s_\theta(f,T) \right)   
\end{align*}
where
\begin{align}
\label{sensformula1}
s_\theta(f,T) =  \sum_{k = 1}^K &\left( \int_{0}^T \frac{ \partial \lambda_k ( X_\theta (t) , \theta ) }{ \partial \theta   } \Delta_{\zeta_k} f(X_\theta(t)) dt  \right. \\ & \left. 
+ \sum_{  i = 0 : \sigma_i < T   }^{\infty}  \frac{  \partial  \lambda_k ( X_\theta(\sigma_{i}) ,\theta ) }{ \partial \theta}  R_{\theta}( X_\theta( \sigma_i) ,f, T -\sigma_i ,k) \right). \notag
\end{align}
\end{theorem}

The proof of this theorem is given in Section 4. Note that the random variable $s_{\theta}(f,T)$ can be evaluated from the realizations of the Markov process $\{ X_\theta(t) : t \geq 0\}$, provided that we can calculate the values of the function $R_{\theta}$ that are required. In almost all the cases these values cannot be calculated explicitly and they have to be estimated during the simulation run. We shall deal with this issue in detail in the next section. For now assume that we ``know" the function $R_\theta$ and hence using independent realizations of $\{ X_\theta(t) : t \geq 0\}$ we can generate $N$ independent samples $s^{(1)}_\theta(f,T) , \dots, s^{(N)}_\theta(f,T) $ from the distribution of $s_\theta(f,T) $. Then $S_{\theta}(f,T)$ can be estimated as
\begin{align}
\label{sens_est1}
\hat{S}_\theta(f,T) = \frac{1}{N} \sum_{i = 1}^N s^{(i)}_{\theta}(f,T) 
\end{align}•
and the variance of this estimator is
\begin{align}
\label{sens_est_var}
\mathrm{Var}(\hat{S}_\theta(f,T)) = \mathrm{Var} \left( \frac{1}{N} \sum_{i = 1}^N s^{(i)}_\theta(f,T)  \right) = \frac{1}{N} \mathrm{Var}(s_\theta(f,T) ).
\end{align}•
Generally statistical quantities are estimated along with a confidence interval which indicates the accuracy of the estimate. The half-length of the $95 \%$ confidence interval for the above estimator is given by $1.96  \sqrt{\mathrm{Var}(\hat{S}_\theta(f,T)) } =  1.96  \ \sqrt{\mathrm{Var}(s_\theta(f,T)) }/ \sqrt{N} $. Since the variance of $s_\theta(f,T)$ is unknown we can substitute it by the sample variance $\hat{v}(s_\theta(f,T))$ evaluated as
\begin{align}
\label{formulaforsamplevariance}
\hat{v}(s_\theta(f,T)) =  \frac{1}{N-1} \sum_{i = 1}^N  \left(  s^{(i)}_\theta(f,T)  - \hat{S}_\theta(f,T)  \right)^2.
\end{align}
Note that the number of samples ($N$) needed for estimating $\hat{S}_\theta(f,T)$ within a certain confidence interval, is directly proportional to the variance of $s_\theta(f,T)$.

Now we briefly discuss the Girsanov method presented in \cite{Gir} and the problems associated with it. Suppose that the sensitive parameter $\theta$ appears linearly in only one propensity function $\lambda_{k_0}$ for some $k_0 \in \{1,\dots,K\}$. In this case, it is shown in \cite{Gir} that $S_{\theta}(f,T) = \E (s_\theta(f,T))$ with $s_\theta(f,T)$ given by
\begin{align}
\label{form_gir}
s_\theta(f,T) = \frac{ f(X_\theta(T)) M^{k_0}_\theta(T)  }{ \theta },
\end{align}•
where
$M^{k_0}_\theta(t) = ( N^{k_0}_\theta(t) - \int_{0}^{t} \lambda_{k_0}(X_\theta(s), \theta) ds)$ is a martingale
and $N^{k_0}_\theta(t)$ is the number of times reaction $k_0$ fired until time $t$. Observe that this formula cannot be used to determine $S_{\theta}(f,T)$ for $\theta = 0$ even though $S_{\theta}(f,T)$ (as defined by \eqref{sensitivity1}) makes perfect sense at $\theta = 0$. Moreover, as the next example illustrates, the random variable $s_\theta(f,T)$ (given by \eqref{form_gir}) can have a large variance for small values of $\theta$ making it necessary to take very large sample sizes for estimating accurately.
\begin{example}
\label{example1}
{ \rm
Consider a pure birth process in which a chemical species $\mathcal{S}$ is created at rate $\theta$
\begin{align*}
\emptyset \stackrel{\theta}{\rightarrow} \mathcal{S}.
\end{align*}
The population at time $t$ is given by $X_\theta(t) \in \N_0$. We assume that $X_{\theta}(0) = 0$. There is only reaction, with propensity function $\lambda_1(x,\theta) = \theta$ and stoichiometric vector $\zeta_1 = 1$.
For this system we can write $X_{\theta}(t) = Y(\theta t)$ where $Y$ is a unit rate Poisson process. Let $f : \N_0 \to \R$ be given by $f(x) = x$. Then for any $T \geq 0$
\begin{align}
\label{sens_example}
S_\theta(f,T) = \frac{d}{d \theta} \E( f(X_\theta(T)) ) = \frac{d}{d \theta} \E( X_{\theta}(T) )  =  \frac{d}{d \theta}  \E (Y(\theta T))  =\frac{d (\theta T) }{d \theta}   = T.
\end{align}
The last equality follows from the fact that $Y(\theta T)$ is a Poisson random variable with rate $\theta T$. 
According to the Girsanov method, $S_\theta(f,T) = \E \left( s_\theta(f,T)\right)$ where $s_\theta(f,T)$ can be evaluated using \eqref{form_gir} as
\begin{align*}
s_\theta(f,T)= \frac{1}{\theta} Y_1(\theta T)\left( Y_1(\theta T) - \theta T  \right).
\end{align*}
Using the moments of Poisson random variables one can verify that as expected, $\E \left( s_\theta(f,T)\right) = T$,  but its variance is
\begin{align*}
\textrm{Var}\left(s_\theta(f,T)\right) =  \frac{ T  + 4 \theta T^2 + \theta^2 T^3}{\theta}.
\end{align*}
This shows that for $\theta \approx 0$, $\textrm{Var}\left(s_\theta(f,T)\right)$ is very high and hence the Girsanov estimator will perform poorly. Also for $\theta = 0$ we cannot determine $S_\theta(f,T) $ with this method even though $S_\theta(f,T)$ is well-defined, as shown by \eqref{sens_example}. 

Recall the definition of $\Psi_\theta$ and $R_\theta$ from \eqref{defpsixtheta} and \eqref{defrktheta}. For this example, $\Psi_\theta(x,f,t) =x +\theta t$ and hence $R_\theta(x,f,t,1) = 0$ for all $x$ and $t$. Therefore Theorem \ref{thm:sens} says that $S_\theta(f,T) = \E\left( s_\theta(f,T)\right)$ with $s_\theta(f,T) = T$ (see \eqref{sensformula1}). Hence the estimator based on Theorem \ref{thm:sens} has variance $0$ ! Moreover this estimator works for $\theta = 0$, unlike the Girsanov method. 
 } \hfill  $\square$
\end{example}

Note that for small values of $\theta$, the reaction in Example \ref{example1} will have very few firings in the observation time period $[0,T]$. 
Hence this example indicates that if the sensitive parameter $\theta$ is the rate constant of a reaction which is \emph{slow} in the reference time-scale of the system, then the Girsanov method can be highly inefficient, while an estimator based on Theorem \ref{thm:sens} can perform much better. We illustrate this through another example.
\begin{example}
\label{example_bd}{ \rm {\bf (Single-species birth-death model) : }
Consider the process in which a chemical species $\mathcal{S}$ is created and destroyed according to the following two reactions:
\begin{align*}
\emptyset \stackrel{1 }{\rightarrow} \mathcal{S} \stackrel{\theta }{\rightarrow} \emptyset.
\end{align*}
The population at time $t$ is given by $X_\theta(t) \in \N_0$. Conditioned on $X_\theta(t) = x$, the propensity functions for the first and second reactions are $\lambda_1(x,\theta) = 1 $ and $\lambda_2(x, \theta) = \theta x$ respectively. We assume that $X_\theta(0) = 0$. Let $f(x) = x$, then 
\begin{align*}
\Psi_{\theta}(x,f,t) = \mathbb{E} \left( X_\theta(t) \middle\vert X_{\theta}(0) = x\right) = x e^{ - \theta  t} + \frac{1}{  \theta }  (1 - e^{ - \theta t}  ).
\end{align*}
This allows us to compute
\begin{align}
\label{rktheta_example}
R_{\theta}(x,f,t,2) & =  \int_{0}^{t} \left(     1   -  e^{ - \theta s}\right) e^{ -  (  1 + \theta x ) (t-s)} ds  = \left(  \frac{ 1 - e^{ -( 1 + \theta  x )t }  }{ 1+ \theta x  }  \right) - \left(  \frac{ e^{- \theta  t } - e^{ -( 1 + \theta x )t }  }{ 1 + \theta (x - 1)  }  \right).
\end{align}
According to Theorem \ref{thm:sens} we have $S_{\theta}(f,T)  =  \E (s_\theta(f,T))$ where 
\begin{align}
\label{bdex_stheta}
s_\theta(f,T) = - \int_{0}^T   X_\theta (t) dt 
+ \sum_{  i = 0 : \sigma_i < T   }^{\infty}  X_\theta( \sigma_i)  R_{\theta}( X_\theta( \sigma_i) ,f, T -\sigma_i ,2).
\end{align}
Equation \eqref{rktheta_example} gives us the explicit formula for $R_\theta$. With this in our hands, we can generate samples from the distribution of $s_\theta(f,T)$ and compute the sample variance 
$\hat{v}(s_\theta(f,T))$ (see \eqref{formulaforsamplevariance}). We can do the same for $s_\theta(f,T)$ given by the Girsanov method (see \eqref{form_gir}). 
\begin{table}[h]
\caption{Comparison of sample variances}
\label{table:examplebd}
\centering
\begin{tabular}{|c | c | c | c | c | c | }
\hline
& & \multicolumn{4}{|c| }{ $T$ }    \\    \cline{3-6} 
$\theta$ & Method  & 1 & 5 & 10 & 20 \\ \hline 
\multirow{2}{*}{$0.1$} & Girsanov & 10.7365 & 2303.39 & 20698 & 112758   \\
& Our & 0.2905 & 20.473 & 90.8017 & 326.391 \\ \hline
\multirow{2}{*}{$0.01$} & Girsanov & 99.6366 & 33719.1 & 489925 & $6.6203 \times 10^{6}$   \\
& Our & 0.3343 & 37.6357 & 281.547 & 1923.5 \\ \hline
\multirow{2}{*}{$0.001$} & Girsanov & 302.818 & 373393 & $5.76119\times 10^{6}$ & $7.81587\times 10^{7}$  \\
& Our & 0.3447 & 41.6996 & 329.948 & 2519.29 \\ \hline
\multirow{2}{*}{$0.0001$} & Girsanov & 10004.2 & $3.85596\times 10^{6}$ & $6.50532\times 10^{7}$ & $7.99393\times 10^{8}$  \\
& Our & 0.3364 & 41.1659 & 334.106 & 2620.81 \\ \hline
\end{tabular}
\end{table}
In Table \ref{table:examplebd} we present a comparison of the sample variances obtained by both these methods for $\theta = \{ 0.1, 0.01,0.001,0.0001 \}$ and $T = \{ 1,5,10,20 \}$. 
These results allow us to make the following observations. For all the values of $\theta$ and $T$ that are considered, our method gives a much lower sample variance than the Girsanov method. 
Moreover the variances obtained by the Girsanov method increase significantly as $\theta$ decreases in size, while for our method they remain the same. We discussed before that the efficiency of an estimator is inversely proportional to its variance. The results in Table \ref{table:examplebd} reinforce our claim about the inefficiency of the Girsanov method while estimating the sensitivity with respect to the rate constants of \emph{slow} reactions. In such situations, the estimator based on Theorem \ref{thm:sens} can be more efficient.
} \hfill $\square$
\end{example}

\section{Algorithm for estimation of the parameter sensitivity}

In Example \ref{example_bd}, we had an analytical expression for the function $R_{\theta}$ that allowed us to generate samples of $s_\theta(f,T)$ using \eqref{sensformula1}. For most examples of interest, we would not have this luxury. Therefore we need to find a way to estimate the required values of the function $R_{\theta}$ ``on the run". We deal with this issue now and present an algorithm to generate samples for estimating the parameter sensitivity. Note that the method we propose is just one way to estimate the sensitivity using Theorem \ref{thm:sens}. There could be better approaches, perhaps using smarter data structures, that can do the job more efficiently. We encourage the research community to explore this issue.

Let $\{ X_\theta(t) : t \geq 0\}$ be a Markov process with generator $\mathbb{A}_\theta$ and initial state $x_0$. Its jump times are given by $\sigma_0,\sigma_1,\dots$, where $\sigma_0 = 0$.
Our goal is to estimate the quantities of the form $R_{\theta}( X_\theta( \sigma_i) ,f, T -\sigma_i ,k)$ that appear in \eqref{sensformula1} and then evaluate $s_\theta(f,T)$.
Moreover to preserve the unbiasedness of our method, we require these estimates to be unbiased as well.
Estimating several such quantities in parallel is quite challenging and hence the algorithm we are about to present is much harder to implement than the Girsanov method. However once implemented, our algorithm can offer considerable speed-ups over the Girsanov method in certain situations. We shall demonstrate this later through examples. We first make the formula \eqref{sensformula1} more amenable for calculations.

Define
\begin{align}
\label{defeta}
\eta = \max \{ i \geq 0 :  \sigma_i < T \}
\end{align}
and let
\begin{align}
\label{defdeltat}
\Delta t_i  = \left\{
\begin{array}{cc}
 (\sigma_{i+1} - \sigma_i) & \textrm{ for } i = 0,\dots,\eta-1  \\
 (T - \sigma_\eta) & \textrm{ for } i = \eta . \\
\end{array}\right.
\end{align}
Then
\begin{align}
\label{defn_integral}
 \int_{0}^T \frac{ \partial \lambda_k ( X_\theta (t) , \theta ) }{ \partial \theta   } \Delta_{\zeta_k} f(X_\theta(t)) dt  = \sum_{ i = 0}^{\eta}  
\frac{ \partial \lambda_k ( X_\theta (  \sigma_i  ) , \theta ) }{ \partial \theta   } \Delta_{\zeta_k} f(X_\theta(  \sigma_i  )) \Delta t_i.
\end{align}
From \eqref{defrktheta} we know that evaluating $R_\theta$ requires us to compute exponentially weighted integrals, which can be cumbersome. To avoid this issue we do the following. 
If $\lambda_0(X_\theta( \sigma_i),\theta) > 0$, then let $\gamma_i$ be an exponential random variable with rate $\lambda_0(X_\theta( \sigma_i),\theta)$ and let 
\begin{align}
\label{defalphai}
\alpha_i = (T - \sigma_i - \gamma_i)^{+}.
\end{align}
Conditioned on $\lambda_0(X_\theta( \sigma_i),\theta)$, the random variable $\gamma_i$ is independent of everything else. Then for any $\zeta \in \N_0^d$
\begin{align*}
& \E \left(  \int_{0}^{T - \sigma_i} \Psi_\theta( X_\theta(\sigma_i) +\zeta, f,s  ) e^{- \lambda_0(X_\theta(\sigma_i) ,\theta) ( T - \sigma_i - s ) }  ds \middle \vert X_\theta(\sigma_i), \sigma_i  \right) 
\\& = \E \left( \Psi_\theta ( X_\theta(\sigma_i) +\zeta, f,\alpha_i   ) \middle \vert X_\theta(\sigma_i), \sigma_i    \right).
\end{align*}•
Hence for any $k = 1,\dots,K$ we have  
\begin{align}
\label{rtheta_non_absorbing}
 &  R_{\theta}( X_\theta( \sigma_i) ,f, T -\sigma_i ,k)  =  \frac{1}{ \lambda_0( X_\theta (\sigma_i) ,\theta)  }   \Big(  \E \left(  \Psi_{\theta}( X_\theta (\sigma_i) + \zeta_k,f, \alpha_i ) \middle \vert X_\theta( \sigma_i), \sigma_i  \right)  \notag \\ &  - \E \left(  \Psi_{\theta}( X_\theta (\sigma_i) ,f, \alpha_i ) \middle \vert X_\theta( \sigma_i), \sigma_i  \right)
  - \Delta_{\zeta_k} f(X_\theta (\sigma_i)) \Big). 
\end{align}•
On the other hand if $\lambda_0(X_\theta( \sigma_i),\theta) = 0$, then $X_\theta( \sigma_i)$ is an absorbing state for the Markov proces, which implies that 
$\Psi_\theta ( X_\theta (\sigma_i) ,f, t ) = f(X_\theta (\sigma_i))$ for all $t \geq 0$. Therefore 
\begin{align}
\label{rtheta_absorbing}
&  R_{\theta}( X_\theta( \sigma_i) ,f, T -\sigma_i ,k)  =  I_\theta( X_\theta (\sigma_i) +\zeta_k ,f, T - \sigma_i  )    - (T - \sigma_i) f(X_\theta (\sigma_i) +\zeta_k ) ,
\end{align}
where 
\begin{align}
\label{defixft}
I_\theta(x,f,t) = \int_{0}^{t} \Psi_\theta (x,f,s)ds.
\end{align}
Let $\hat{\Psi}_\theta( X_\theta (\sigma_i) + \zeta_k,f, \alpha_i), \hat{\Psi}_ \theta( X_\theta (\sigma_i),f, \alpha_i)$ and $\hat{I}_\theta( X_\theta (\sigma_i) +\zeta_k ,f, T - \sigma_i  )$ be unbiased estimators for $\Psi_\theta( X_\theta (\sigma_i) + \zeta_k,f, \alpha_i), \Psi_\theta( X_\theta (\sigma_i),f, \alpha_i)$ and $I_\theta( X_\theta (\sigma_i) +\zeta_k ,f, T - \sigma_i  )$ respectively. Assume that given $X_\theta( \sigma_i)$ and $ \sigma_i$, these estimators are independent of the sigma field $\mathcal{F}_{\sigma_i}$, where $\{ \mathcal{F}_t\}$ is the filtration generated by $\{ X_\theta(t) : t \geq 0 \}$. If we define
\begin{align*}
& \hat{R}_\theta(X_\theta( \sigma_i) ,f, T -\sigma_i ,k)  \\
&  = \left\{
\begin{array}{cc}
 \frac{ \hat{\Psi}_{\theta}( X_\theta (\sigma_i) + \zeta_k,f, \alpha_i )  -   \hat{\Psi}_{\theta}( X_\theta (\sigma_i) ,f, \alpha_i )   - \Delta_{\zeta_k} f(X_\theta (\sigma_i)) } { \lambda_0(X_\theta (\sigma_i),\theta) } & \textrm{ if }  \lambda_0(X_\theta (\sigma_i),\theta) > 0  \\
 \hat{I}_\theta( X_\theta (\sigma_i) +\zeta_k ,f, T - \sigma_i  )    - (T - \sigma_i) f(X_\theta (\sigma_i) +\zeta_k )  & \textrm{ otherwise } , \notag
\end{array}
\right.
\end{align*}
then due to \eqref{rtheta_non_absorbing} and \eqref{rtheta_absorbing} we must have
\begin{align}
\label{unbiasedestm1}
R_{\theta}( X_\theta( \sigma_i) ,f, T -\sigma_i ,k)  = \E \left(   \hat{R}_{\theta}( X_\theta( \sigma_i) ,f, T -\sigma_i ,k)  \middle \vert X_\theta( \sigma_i), \sigma_i  \right) .
\end{align}• 
In other words, given $X_\theta( \sigma_i)$ and $ \sigma_i$, $\hat{R}_{\theta}( X_\theta( \sigma_i) ,f, T -\sigma_i ,k) $ is an unbiased estimator for $R_{\theta}( X_\theta( \sigma_i) ,f, T -\sigma_i ,k)$. 
Observe that if $\lambda_0(  X_\theta (  \sigma_i  ) , \theta ) = 0$ for some $i$, then $X_\theta (  \sigma_i )$ is an absorbing state and hence the next jump time $\sigma_{i+1} = \infty$. Therefore for $i = 0,\dots,\eta-1$, $ X_\theta (  \sigma_i  )$ can never be an absorbing state. 

Define $\hat{s}_\theta(f,T)$ as
\begin{align}
\label{sthetahat:form2}
 \hat{s}_\theta(f,T) & = \sum_{k = 1}^K  \left[  \sum_{ i = 0}^{\eta-1}  
\frac{ \partial \lambda_k ( X_\theta (  \sigma_i  ) , \theta ) }{ \partial \theta   } \Delta_{\zeta_k} f(X_\theta(  \sigma_i  )) \left( \Delta t_i - \frac{1}{  \lambda_0(  X_\theta (  \sigma_i  ) , \theta )  } \right)  
\right.  \\
& \left. + \sum_{i = 0}^{\eta - 1} \frac{1}{  \lambda_0(  X_\theta (  \sigma_i  ) , \theta )   } \frac{ \partial \lambda_k ( X_\theta (  \sigma_i  ) , \theta ) }{ \partial \theta   } \left(  \hat{\Psi}_\theta ( X_\theta (\sigma_i) + \zeta_k, f ,\alpha_i   ) -  \hat{\Psi}_\theta (X_\theta (\sigma_i) , f ,\alpha_i )  \right)  \right.  \notag \\
& \left.  + \beta\frac{ \partial \lambda_k ( X_\theta (  \sigma_\eta  ) , \theta ) }{ \partial \theta   } \Delta_{\zeta_k} f(X_\theta(  \sigma_\eta  )) \left( \Delta t_\eta - \frac{1}{  \lambda_0(  X_\theta (  \sigma_\eta  ) , \theta )  } \right) \right. \notag \\
& \left.  + \beta \frac{1}{  \lambda_0(  X_\theta (  \sigma_\eta  ) , \theta )   } \frac{ \partial \lambda_k ( X_\theta (  \sigma_\eta  ) , \theta ) }{ \partial \theta   } \left(  \hat{\Psi}_\theta ( X_\theta (\sigma_\eta) + \zeta_k, f ,\alpha_\eta   ) -  \hat{\Psi}_\theta (X_\theta (\sigma_\eta) , f ,\alpha_\eta )  \right)  \right. \notag \\
& \left. +(1 - \beta) \frac{ \partial \lambda_k ( X_\theta (  \sigma_\eta  ) , \theta ) }{ \partial \theta   } \left(  \hat{I}_\theta ( X_\theta (\sigma_\eta) + \zeta_k, f ,\Delta t_\eta   ) -  (\Delta t_\eta ) f( X_\theta (\sigma_\eta) )  \right) \right], \notag
\end{align}
where 
\begin{align*}
\beta= \left\{
\begin{array}{cc}
 1 & \textrm{ if }\lambda_0(  X_\theta (  \sigma_\eta  ) , \theta ) > 0 \\
 0 & \textrm{ otherwise }. \\
\end{array}\right.
\end{align*}
Since $s_\theta(f,T)$ is given by \eqref{sensformula1}, the relations \eqref{defn_integral} and \eqref{unbiasedestm1}  imply that 
$\E (s_\theta(f,T)) = \E ( \hat{s}_\theta (f,T))$. Hence due to Theorem \ref{thm:sens} we have
\begin{align}
\label{sensformula3}
S_\theta(f,T) = \E (\hat{s}_\theta (f,T)).
\end{align}
which shows that we can estimate $S_\theta(f,T)$ by computing the sample mean of several independent realizations of the random variable $\hat{s}_\theta(f,T)$. Now we outline an algorithm to obtain such  realizations.

Observe that the evaluation of $\hat{s}_{\theta}(f,T)$ requires us to estimate either $\Psi_\theta( X_\theta (\sigma_i) + \zeta_k,f, \alpha_i)$ and $\Psi_\theta( X_\theta (\sigma_i),f, \alpha_i)$ or $I_\theta( X_\theta (\sigma_i) +\zeta_k ,f, T - \sigma_i )$ for many values of $i$ and $k$. 
These quantities can be estimated using several paths of the Markov process with generator $\mathbb{A}_\theta$ and initial states $X_\theta (\sigma_i) + \zeta_k$ or $X_\theta (\sigma_i)$. 
If we generate such paths independently for each $i$ and $k$, then the problem quickly becomes computationally intractable. So we need another way. Generally a Markov process representing chemical kinetics is such that its various paths visit the same states again and again. We can use this observation to our advantage and try to estimate all the required quantities with the same set of paths. For this purpose, we do the following. In addition to the \emph{base} path of the Markov process we also independently generate $M$ \emph{auxiliary} paths with the same initial state. From the \emph{base} path we determine what quantities need to be estimated and with the help of the \emph{auxiliary} paths we carry out the estimation. Note that if the state $x$ is visited by $m$ paths $X_1,\dots,X_m$ at times $t_1,\dots,t_m$, then $\Psi_\theta(x,f,t)$ and $I_\theta(x,f,t)$ can be estimated as
\begin{align}
\label{howtoestimate}
\hat{\Psi}_\theta (x ,f,t) = \frac{1}{m} \sum_{i = 1}^m f( X_i( t_i + t ) )  \textrm{  and  }   \hat{I}_\theta(x,f,t) = \frac{1}{m} \sum_{i = 1}^m \int_{0}^{t} f( X_i( t_i + s ) )ds  .
\end{align}• 
If there is no path that reaches $x$ then $\Psi_\theta(x,f,t)$ or $I_\theta(x,f,t)$ cannot be estimated this way. In this case, we estimate such a quantity using a single independently generated path. Once all the required quantities are estimated we can obtain a realization of $\hat{s}_\theta(f,T)$ using \eqref{sthetahat:form2}. 

A rough sketch of the method we just described is given below.
\begin{enumerate}
\item Generate a \emph{base} path of the Markov process.
\begin{itemize}
\item Store all the quantities in \eqref{sthetahat:form2} that can be directly calculated from the \emph{base} path.
\item Create a list $\mathcal{L}$ with all the quantities of the form  $\Psi_\theta (x,f,t)$ or $I_\theta(x,f,t)$ that need to be estimated.
\end{itemize}
\item Generate $M$ \emph{auxiliary} paths of the Markov process. 
\begin{itemize}
\item Use these paths to estimate the quantitites in $\mathcal{L}$.
\item If a quantity cannot be estimated with these paths, then estimate it with a single independently generated path.
\end{itemize}
\item Use the expression \eqref{sthetahat:form2} to evaluate $\hat{s}_\theta(f,T)$.
\end{enumerate}
All the paths can be generated using Gillespie's \emph{Stochastic Simulation Algorithm} \cite{GP}. Since our method relies on the process visiting the same states again and again, we simulate the paths until time $\kappa T$ (rather than $T$), where $\kappa > 1$ is an \emph{extension factor}. The parameters $M$ and $\kappa$ can be used to fine-tune the performance of our method. Note that our method would produce an unbiased estimate of the parameter sensitivity irrespective of the choice of $M$ and $\kappa$.  

In implementing the method sketched above, one has to deal with many issues. It is important to store all the information properly and perform all the book-keeping that is necessary to estimate the quantitites in $\mathcal{L}$. The efficiency of our method depends crucially on how we store the list $\mathcal{L}$. It has to be done in such a way so that when a path being simulated reaches a state $x$, one can easily determine if a quantity of the form $\Psi_\theta (x,f,t)$ or $I_\theta(x,f,t)$ is inside $\mathcal{L}$. In our implementation we use a \emph{Hashtable} (see \cite{CormenAlgo}) for this purpose, where the \emph{hashing} function maps the states in $\N^d_0$ to the indices of a large array. The full details of our implementation are contained in \cite{report}.

From now on we shall refer to the method outlined above as APA, which is an acronym for the \emph{Auxiliary Path Algorithm}. Even though this method is designed to compute the sensitivity with respect to a single scalar parameter $\theta$, we can easily extend it to compute sensitivity with respect to many parameters simultaneously. Even though APA is harder to program than the Girsanov method and has larger memory requirements\footnote{APA requires memory of a size which scales linearly with the total number of jumps in the Markov process within the time period $[0, \kappa T]$. See \cite{report} for details.}, it can be very useful in certain situations. We illustrate this through a couple of examples. For both the examples, we set $M$ (number of \emph{auxiliary} paths) to $50$ and $\kappa$ (\emph{extension factor}) to $3$. In all our numerical examples, we will estimate the sensitivity $\hat{S}_\theta(f,T)$ using the minimum sample size $N$ that is needed to ensure that the half-length of the $95 \%$ confidence interval is below $5 \%$ of $| \hat{S}_\theta(f,T) | $, where $|\cdot|$ is the absolute value function. In our results, the sensitivity estimate $\hat{S}_\theta(f,T)$ will be written in the form $s \pm l$ which means that the $95\%$ confidence interval is equal to $[s - l, s+l]$. While presenting the results we always indicate the CPU time\footnote{All the computations in this paper were performed using C++ programs on an Apple machine with a 2.2 GHz Intel i7 processor.} (in seconds) that was required for the estimation. The CPU time can be taken as a measure of the efficiency of a given method.
\begin{example}
\label{example_bd2}{ \rm {\bf (Single-species birth-death model) : } Let us revisit Example \ref{example_bd} of a simple birth-death process. In that example we ``cheated" in the sense that we used an exact expression for the function $R_\theta$ (see \eqref{rktheta_example}) to generate the samples for sensitivity estimation. Hence we did not have to go through the complex estimation procedure that is required by APA. To compare the performance of the Girsanov method and APA for this example, we present the results for sensitivity estimation obtained by both these methods in Table \ref{table:examplebd2}. These results are provided for $\theta = \{0.1,0.01,0.001,0.0001\}$ and $T = \{5,10 \}$.
\begin{table}[h]
\caption{Comparison of results for the Birth-Death model} 
\label{table:examplebd2} 
\centering
\begin{tabular}{|c | c | c | c | c | c |}
\hline
$\theta$ & T &  Method  & $\hat{S}_\theta(f,T)$ & N & CPU time (s)    \\ \hline 
\multirow{4}{*}{$0.1 $} 
& \multirow{2}{*}{5} & Girsanov & $-8.9671  \pm 0.4483$ & 46271& 0.0449 \\
& & APA &   $-8.8021  \pm 0.4400 $ &  448 &  0.1094 \\ \cline{2-6}
& \multirow{2}{*}{10} & Girsanov &  $-25.7869 \pm  1.2893 $ & 45885 & 0.0851  \\
& & APA &    $-26.5226  \pm  1.3231  $ & 270 & 0.1579  \\  \hline \hline

\multirow{4}{*}{$0.01 $} 
& \multirow{2}{*}{5} & Girsanov & $12.1866  \pm 0.6093 $ & 370157 & 0.2816 \\
& & APA & $-12.2895  \pm  0.6142 $ & 384   & 0.0716\\ \cline{2-6}
& \multirow{2}{*}{10} & Girsanov &  $-47.3787 \pm 2.3689 $ & 334872 & 0.4460\\
& & APA &  $-46.5443   \pm 2.3234 $ & 231   & 0.0895  \\ \hline \hline

\multirow{4}{*}{$0.001$}
& \multirow{2}{*}{5} & Girsanov & $-12.5555  \pm  0.6278 $ & $3.62 \times 10^6 $ &  2.7242  \\
& & APA & $-12.4529 \pm  0.6218$ & 391   & 0.0701  \\ \cline{2-6}
& \multirow{2}{*}{10} & Girsanov &  $ -49.132   \pm  2.4566$ & $3.46 \times 10^6$  & 4.5133  \\
& & APA &  $ -50.823\pm  2.5409  $ & 223  & 0.0785 \\ \hline \hline

\multirow{4}{*}{$0.0001$} 
& \multirow{2}{*}{5} & Girsanov & $-12.7847 \pm 0.6392$ & $3.50 \times 10^7$ & 25.8511 \\
& & APA & $-12.6859  \pm 0.6333  $ & 449 & 0.0773 \\ \cline{2-6}
& \multirow{2}{*}{10} & Girsanov & $-50.6981\pm  2.5349 $ & $ 3.28  \times 10^7 $ & 41.9501   \\
& & APA &  $-50.0198  \pm  2.4969 $ & 249 &  0.0847 \\ \hline \hline

\end{tabular}
\end{table}

From Table \ref{table:examplebd2} we can make the following observations. Unlike the Girsanov method, the performance of APA stays the same as $\theta$ decreases in magnitude. APA is slightly slower than the Girsanov method for $\theta = 0.1$, but faster for all other values of $\theta$. In fact for $\theta = 0. 001$ and $\theta = 0.0001$, APA is more efficient than the Girsanov method by a factor of more than $10$ and $100$ respectively. 
} \hfill $\square$
\end{example}

\begin{example}
\label{example_ge}{ \rm {\bf (Gene Expression Network) : } 
We now consider the model for gene transcription and translation that appeared in \cite{MO}. It has three species : Gene ($G$), mRNA ($M$) and protein ($P$), and there are four reactions given by
\begin{align*}
G \stackrel{ k_R }{\rightarrow} G + M , \  \  M \stackrel{ k_P }{\rightarrow} M + P,  \  \ M \stackrel{ \gamma_R }{\rightarrow} \emptyset  \   \textrm{ and } P \stackrel{ \gamma_P}{\rightarrow} \emptyset.
\end{align*}
The first two reactions represent the translation of a single gene into mRNA and the transcription of mRNA into proteins. The final two reactions are degradation of mRNA and protein molecules. The rate constants for translation, transcription, mRNA degradation and protein degradation are $k_R$, $k_P$, $\gamma_R$ and $\gamma_P$ respectively.  Typically, a mRNA molecule decays a lot faster than a protein molecule. The half-life of the former is usually in minutes (or seconds) while the half-life of the latter is usually in hours. We shall fix $k_R = 0.6 \ \mathrm{min}^{-1}$, $k_P = 1.7329 \ \mathrm{min}^{-1}$ and $\gamma_R = 0.3466 \ \mathrm{min}^{-1}$. These values are given in \cite{MO} for \emph{lacA} gene in \emph{E.Coli}. Our sensitive parameter $\theta$ is the protein degradation rate $\gamma_P$.

The state of this system at time $t$ is $X_\theta(t) = (X_{\theta,1}(t) , X_{\theta,2}(t))$, where $X_{\theta,1}(t)$ and $X_{\theta,2}(t)$ are the number of mRNA and protein molecules respectively. We assume that 
$(X_{\theta,1}(0) , X_{\theta,2}(0)) = (0,0)$. Define $f : \N^2_0 \to \R$ by $f(x_1,x_2) = x_2$. We would like to estimate
\begin{align}
\label{sens_example2}
S_\theta(f,T) = \frac{\partial  }{\partial \theta} \E \left( f(X_\theta(T)) \right) = \frac{\partial  }{\partial \theta} \E ( X_{\theta,2}(T) ) .
\end{align}

We consider five values of $\theta$ : $0.0693 \ \mathrm{min}^{-1}$, $0.0116 \ \mathrm{min}^{-1}$, $0.0023 \ \mathrm{min}^{-1}$, $0.0012 \ \mathrm{min}^{-1}$ and $0$. These correspond to the protein half-life of $10 \ \mathrm{min}$, $1\ \mathrm{hr}$, $5 \ \mathrm{hr}$, $10 \ \mathrm{hr}$ and $\infty$. In Table \ref{table:examplege} we estimate the sensitivity \eqref{sens_example2} using APA and the Girsanov method for two values of $T$ :  $5  \ \mathrm{min}$ and $10 \  \mathrm{min}$. Of course for $\theta = 0$, the Girsanov method cannot be applied and so the values are only given for APA. 
\begin{table}[h]
\caption{Comparison of results for the Gene Expression Network} 
\label{table:examplege} 
\centering
\begin{tabular}{|c | c | c | c | c | c |}
\hline
$\theta$ & T &  Method  & $\hat{S}_\theta$ & N & CPU time (s)    \\ \hline 
\multirow{4}{*}{$0.0693 $} 
& \multirow{2}{*}{5} & Girsanov & $-12.1080  \pm  0.6054 $ & 331945 & 0.6330 \\
& & APA & $ -12.2757 \pm 0.6136 $ &  1913 & 1.1284  \\ \cline{2-6}
& \multirow{2}{*}{10} & Girsanov &  $-61.3473 \pm   3.0670 $ & 253048 & 1.2661  \\
& & APA & $-61.2132 \pm  3.0589   $ & 1277  & 1.9463 \\ \hline \hline

\multirow{4}{*}{$0.0116 $} 
& \multirow{2}{*}{5} & Girsanov & $-13.878 \pm 0.6939  $ & 1975557& 3.5028 \\
& & APA & $-13.821 \pm  0.6908 $ & 1943   & 0.9470\\ \cline{2-6}
& \multirow{2}{*}{10} & Girsanov &  $-80.8423 \pm 4.0420 $ & 1554726 & 6.7256\\
& & APA &  $  -82.7886    \pm 4.1375 $ & 1664   & 1.9465   \\ \hline \hline

\multirow{4}{*}{$0.0023$}
& \multirow{2}{*}{5} & Girsanov & $-14.6221   \pm  0.7311 $ & $9406555 $ &  16.259 \\
& & APA & $ -14.8203 \pm  0.7410 $ & 2384  & 1.1224 \\ \cline{2-6}
& \multirow{2}{*}{10} & Girsanov &  $  -89.1083   \pm  4.4554 $ & $7102591 $  & 29.4725  \\
& & APA &  $-86.6336 \pm  4.3314 $ & 2236 & 2.3528 \\ \hline \hline

\multirow{4}{*}{$0.0012$} 
& \multirow{2}{*}{5} & Girsanov & $ -14.9095\pm 0.7455 $ & $17378930$ & 30.0565 \\
& & APA & $-14.9071   \pm 0.7452$ & 12047 & 0.9351\\ \cline{2-6}
& \multirow{2}{*}{10} & Girsanov &  $-88.5277\pm  4.4264 $ & $13929778$ & 57.6593    \\
& & APA &  $-86.4873 \pm 4.3242 $ & 1919   &  1.9933 \\  \hline \hline

\multirow{2}{*}{$0$}
& 5 & APA &  $-15.037 \pm 0.7518  $ &  1935 & 0.8865 \\ \cline{2-6}
& 10 & APA &  $-83.5049\pm 4.1741 $ & 1797 &  1.8591 \\ \hline \hline
\end{tabular}
\end{table}
The results in Table \ref{table:examplege} have the same underlying message as in Example \ref{example_bd2}. The performance of the Girsanov method deteriorates as $\theta$ gets smaller, while the performance of APA remains unchanged. Apart from the biologically unrealistic case of $\theta = 0.0693 \ \mathrm{min}^{-1}$ (corresponding to protein half-life of $10 \ \mathrm{min}$), APA outperforms the Girsanov method in all other cases. The degree of outperformance is directly proportional to the smallness of $\theta$.
} \hfill $\square$
\end{example}

Examples \ref{example_bd2} and \ref{example_ge} clearly indicate that APA can be much more efficient than the Girsanov method for the estimation of sensitivity with respect to the rate constant of a \emph{slow} reaction in a chemical reaction network.

With so many methods in the literature for estimating the parameter sensitivity, one may wonder which method should be used in a certain situation. We now offer some suggestions. In general if one is willing to tolerate a bias in the sensitivity estimate, then the finite difference schemes given in \cite{KSR1,DA} can be much faster than both the unbiased methods (APA and Girsanov). 
Such schemes estimate the quantity $S_{\theta,h}(f,T)$ (see \eqref{fdmscheme}), by coupling the processes $X_\theta$ and $X_{\theta+h}$ in an intelligent way. 
Three such couplings are : \emph{Common Reaction Numbers} (CRN) (see \cite{KSR1}) , \emph{Common Reaction Paths} (CRP) (see \cite{KSR1}) and \emph{Coupled Finite Differences} (CFD) (see \cite{DA}). The bias introduced by these schemes is proportional to $h$ and one may be want to reduce the bias by picking a very small $h$. However as $h$ gets smaller, the variance of the finite difference estimator gets larger, making these schemes inefficient in comparison to the unbiased methods. 
To illustrate this point we return to Example \ref{example_ge} of the gene expression network. We fix $\theta = 0.0116 \ \mathrm{min}^{-1} $ and $T = 10 \ \mathrm{min}$. The sensitivity value in Example \ref{example_ge} is now estimated using the three finite difference schemes : CRN, CRP and CFD, for four different values of $h$ : $10^{-2}$, $10^{-3}$, $10^{-4}$ and $10^{-5}$. We perform these estimations using the \emph{SPSens} software \cite{software}.  
The results \footnote[2]{In this paper, we always estimate the parameter sensitivity using the minimum number of samples $N$ that are needed to ensure that the half-length of the $95\%$ confidence interval is below $5\%$ of the magnitude of the sensitivity estimate. The current version of the software \emph{SPSens} does not allow us to specify such a stopping criterion. 
Hence by trial and error, we estimate the sensitivity using a sample size $N$ which approximately satisfies this criterion.} are given in Table \ref{tableforfdm}. One can easily verify that the efficiency of each finite difference scheme deteriorates as $h$ gets smaller. From Table \ref{table:examplege} we see that for this choice of $\theta$ and $T$, the Girsanov method and APA estimated the sensitivity in $6.7256 \ \mathrm{sec}$ and $1.9465\ \mathrm{sec}$ respectively. For higher values of $h$, the finite difference schemes (especially CRP and CFD) are much faster than both these methods. However the situation is reversed for the smallest value $h = 10^{-5}$.
\begin{table}[h]
\caption{Efficiency of various finite difference schemes}
\label{tableforfdm}
\centering
\begin{tabular}{|c | c | c | c | c | }
\hline
$h $  &   Method  &$\hat{S}_{\theta,h}$ & N & CPU Time (s) \\ \hline 
\multirow{2}{*}{$10^{-2}$} 
& CRN& $ -82.6147 \pm 4.1478 $  & $ 47500$  & 0.4957 \\
& CRP & $ -81.2400 \pm 4.0207 $  &  2500 & 0.0312   \\ 
& CFD & $ -80.1277 \pm 4.0287$  &  2350  &  0.0229    \\ 
\hline
\multirow{2}{*}{$10^{-3}$} 
& CRN& $ -81.9867 \pm 4.0577 $  & $ 600000$  & 6.4498  \\
& CRP & $ -81.2000  \pm 4.0314 $  &  20000  & 0.2194    \\ 
& CFD & $ -81.7778  \pm 4.1345 $  &  18000  & 0.1498    \\ 
\hline
\multirow{2}{*}{$10^{-4}$} 
& CRN& $ -80.1203 \pm 4.0080 $  & $ 6400000$  & 68.2093  \\
& CRP & $ -83.3846  \pm 4.0611 $  &  195000  & 2.1016   \\ 
& CFD & $ -81.7895  \pm 4.0499 $  &  190000  & 1.5559    \\ 
\hline
\multirow{2}{*}{$10^{-5}$} 
& CRN& $ -83.0656 \pm 4.0350 $  & $ 64000000$  & 724.4415  \\
& CRP & $ -82.1538  \pm 4.0239 $  &  1950000  & 21.0040  \\ 
& CFD & $ -83.6316  \pm 4.1103 $  &  1900000  & 15.6029   \\ 
\hline
\end{tabular}

\end{table}
Based on the this discussion, we suggest the following strategy for picking the \emph{best} method for estimating the sensitivity with respect to the rate constant $\theta$ of a particular reaction. If some amount of \emph{unknown} bias can be tolerated, then a finite difference scheme (preferably CRP or CFD) must be used. However if one does not want to sacrifice the accuracy of the estimate, then instead of trying to reduce the bias by picking a very small $h$ in a finite difference scheme, one should use an unbiased method : APA or the Girsanov method. The results in this paper show that APA would be a better choice for small values of $\theta$ while the Girsanov method would be more appropriate for large values of $\theta$. This strategy for method selection is heuristically depicted in Figure \ref{fig:strategy}.
\vspace{10pt}
\begin{figure}[h!]
  \begin{center}
    \includegraphics[height=5.0cm]{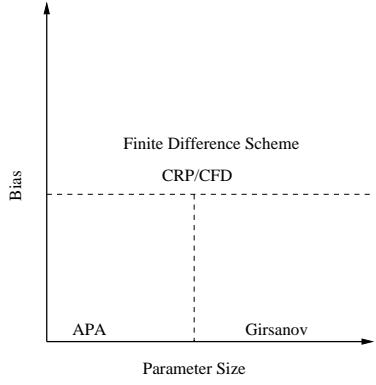}
  \end{center}
  \caption{Strategy for picking the sensitivity estimation method based on the size of the sensitive parameter and the bias that can be allowed.}
\label{fig:strategy}
\end{figure}

\section{Proof of the main result}

In this section we will prove our main result Theorem \ref{thm:sens}. We start by proving a simple lemma.

\begin{lemma}
\label{lemma1}
Let $\{ X_{\theta}(t) : t \geq 0 \}$ be a Markov process with generator $\mathbb{A}_{\theta}$ given by \eqref{genctmc2}. Suppose that the propensity functions $\lambda_1,\dots,\lambda_K$ satisfy Condition \ref{cond:propensityfunctions} at $\theta$. Assume that $\E \left( \left\| X_\theta(0)\right\|^p \right) < \infty$ for all $p \geq 0$. Then we have the following. 
\begin{itemize}
\item[(A)] For any $p\geq 0$ and $T > 0$
\begin{align*}
\sup_{t \in [0,T]} \E \left(  \left\| X_\theta(t) \right\|^p \right) < \infty.
\end{align*}
\item[(B)] If $f : \N^d_0 \to \R$ is a function satisfying Condition \ref{cond:bddness} then for any $t \geq 0$
\begin{align*}
\E \left(  f ( X_\theta(t) ) \right) = \E \left(  f ( X_\theta( 0 ) ) \right) + \E \left(  \int_{0}^{t}   \mathbb{A}_{\theta} f ( X_\theta(s) )    ds\right)
\end{align*}
\end{itemize}
\end{lemma}
\begin{proof}
To prove part (A) of the lemma we can assume that $p$ is a positive integer. Recall the definition of the set $P$ from \eqref{setofpositivereactions}. 
Let 
\begin{align}
\label{defnq}
q= \left\{
\begin{array}{cc}
  \max_{k \in P} \langle \bar{1}, \zeta_k\rangle & \textrm{ if }  P \neq \emptyset\\
  0 & \textrm{ otherwise} .\\
\end{array}\right.
\end{align}
Pick a large $M > 0$ and let $g : \mathbb{N}^d_0 \to \R$ be the function given by
\begin{align*}
g(x) = \|x\|^p \wedge (M + q)^p,
\end{align*}
where $a \wedge b := \min \{a,b\} $.
Since $g$ is bounded, it is in the domain of the generator $\mathbb{A}_{\theta}$.
Define a stopping time $\tau_M$ by
\begin{align*}
\tau_M = \inf \{ t  \geq 0 : \|X_\theta (t)\| \geq M\}.
\end{align*}
Using Dynkin's theorem (see Lemma 19.21 in \cite{Kal}) we can conclude that
\begin{align*}
\E \left( g(X_\theta (t \wedge \tau_M ))\right) &= \E \left( g(X_\theta (0)) \right) + \E \left[  \int_{0}^{t \wedge \tau_M}   \mathbb{A}_{\theta} g(X_\theta (s)) ds  \right] \\
&= \E \left( g(X_\theta (0)) \right) + \E \left[  \int_{0}^{t \wedge \tau_M}   \sum_{k = 1}^K  \lambda_k(X_\theta (s), \theta)  \Delta_{\zeta_k} g(X_\theta (s)) ds  \right].
 \end{align*}
For $0 \leq t < \tau_M$, we have $ \| X_\theta (t) \| < M$, which implies that
\begin{align*}
\E \left( \| X_\theta (t \wedge \tau_M ) \|^p\right) & = \E \left(  \| X_\theta (0)\|^p \right) \\&  + \E \left(  \int_{0}^{t  \wedge \tau_M}  \sum_{k = 1}^K \lambda_k(X_\theta (s),\theta) \left(   \| X_\theta (s) + \zeta_k\|^p - \| X_\theta (s)\|^p \right) ds    \right).
\end{align*}
 Part (C) of Condition \ref{cond:propensityfunctions} ensures that when $\lambda_k(X_\theta (s),\theta)  > 0$, $(X_\theta (s) + \zeta_k) \in \N^d_0$ and hence $\| X_\theta (s) + \zeta_k\| = \langle \bar{1}, X_\theta (s)\rangle + \langle \bar{1}, \zeta_k \rangle$. This gives us
 \begin{align*}
&\E \left( \| X_\theta (t \wedge \tau_M ) \|^p\right) \\& \leq \E \left(  \| X_\theta (0)\|^p \right) + 
\E \left(  \int_{0}^{t\wedge \tau_M }  \sum_{k \in P} \lambda_k(X_\theta (s),\theta) \left(  ( \langle \bar{1} , X_\theta (s) + \zeta_k  \rangle )^p - \langle \bar{1} , X_\theta (s) \rangle^p \right) ds    \right) \\
& \leq  \E \left(  \| X_\theta (0)\|^p \right) + 2^p q^p  \E \left(  \int_{0}^{t }  \sum_{k \in P} \lambda_k(X_\theta (s),\theta) \left(  \| X_\theta (s \wedge \tau_M ) \|^{p-1} + 1 \right) ds    \right).
\end{align*}
Using part (D) of Condition \ref{cond:propensityfunctions} we can find a constant $C >0$ such that
\begin{align*}
\E \left( \| X_\theta (t \wedge \tau_M) \|^p )\right) &\leq  \E \left(  \| X_\theta (0)\|^p   \right) + Ct + C \int_{0}^t  \E \left( \|X_\theta (s \wedge \tau_M)\|^p \right) ds.   
\end{align*}•
From Gronwall's inequality we obtain
\begin{align}
\label{gronwall1}
\E \left( \| X_\theta (t \wedge \tau_M) \|^p )\right) \leq \left( \E ( \| X_\theta (0)\|^p )   + C t \right) e^{C t}.
\end{align}
Using Markov's inequality, for any $t >0$ we get
\begin{align*}
\lim_{M \to \infty} \P \left(   \tau_M <   t  \right) &= \lim_{M \to \infty}\P \left( \| X_\theta (t \wedge \tau_M) \|^p    \geq M^p \right)  \leq  \lim_{M \to \infty}\frac{\E \left( \| X_\theta (t \wedge \tau_M) \|^p )\right) }{M^p}  = 0.
\end{align*}•
The last limit is $0$ due to \eqref{gronwall1}. The above calculation shows that $\tau_M \to \infty$, in probability as $M \to \infty$.
Since $\tau_M$ is monotonically increasing we must have that $\tau_M \to \infty$ a.s. as $M \to \infty$.
Letting $M \to \infty$ in \eqref{gronwall1} and using Fatou's lemma we obtain
\begin{align*}
 \E \left( \| X_\theta (t) \|^p )\right) \leq \lim_{M \to \infty} \E \left( \| X_\theta (t \wedge \tau_M) \|^p )\right) \leq \left( \E ( \| X_\theta (0)\|^p )   + C t \right) e^{C t}.
\end{align*}•
Taking supremum over $t \in [0,T]$ proves part (A) of the lemma. For part (B), note that if $f$ satisfies Condition \ref{cond:bddness} then due to part (A) we must have that
\begin{align*}
\E \left( | f(X_\theta (t)) | \right) < \infty  \textrm{ and } \E \left(  \int_{0}^{t} | \mathbb{A}_{\theta} f( X_\theta (s)  ) | ds\right) < \infty.
\end{align*}•
We can assume that $f$ is a positive function. Pick a large $M$ and define a bounded function $f_M : \N_0^d \to \R$ by
\begin{align*}
f_M(x) = f(x) \wedge M.
\end{align*}
The function $f_M$ is bounded and using Dynkin's theorem (see Lemma 19.21 in \cite{Kal}) we get
\begin{align*}
\E \left( f_M(X_\theta( t ) ) \right) = \E \left( f_M(X_\theta(0) ) \right) + \E \left( \int_{0}^{t}  \mathbb{A}_{\theta} f_M(X_\theta( s ) ) ds  \right) 
\end{align*}
Taking the limit $M \to \infty$ and using the dominated convergence theorem proves part (B) of the lemma.
\end{proof}

Note that the state space of our Markov process is $\N^d_0$. Endowing this space with the discrete metric makes it a complete and separable metric space. Hence under our topology, any real-valued function on $\N^d_0$ is continuous. We now come to the main proof.

\begin{proof}[Theorem \ref{thm:sens}]
Let $X_\theta$ be the process given by \eqref{tcrep2}, with $X_\theta(0) = x_0$. Pick a $h \in \R$ and for any $x_1,x_2 \in \N_0^d$ define
\begin{align*}
\lambda^{min}_k(x_1,x_2,\theta,h) = \lambda_k(x_1,\theta) \wedge \lambda_k(x_2,\theta+h) \textrm{ for } k \in \{1,2,\dots,K\}
\end{align*}
and let $\lambda^{min}_0(x_1,x_2,\theta,h) = \sum_{k = 1}^K\lambda^{min}_k(x_1,x_2,\theta,h)$. 
Let the processes $\hat{X}_\theta$ and $\hat{X}_{\theta+h}$ be given by the following time change representations.
\begin{align}
\label{coupling1}
\hat{X}_{\theta}(t) &= x_0 + \sum_{k =1 }^K \zeta_k Y_k\left( \int_{0}^{t} \lambda^{min}_k ( \hat{X}_{\theta}(s),\hat{X}_{\theta+h}(s), \theta,h)ds \right) \\
& +\sum_{k = 1}^K \zeta_k Y^{(1)}_k \left(  \int_{0}^{t} \left( \lambda_k (\hat{X}_{\theta}(s),\theta)  -  \lambda^{min}_k (\hat{X}_{\theta}(s),\hat{X}_{\theta+h}(s), \theta,h) \right) ds\right)  \notag \\
\label{coupling2}
\hat{X}_{\theta+h }(t) &= x_0 + \sum_{k = 1}^K \zeta_k Y_k\left( \int_{0}^{t}\lambda^{min}_k (\hat{X}_{\theta}(s),\hat{X}_{\theta+h}(s), \theta,h)ds \right) \\
& + \sum_{k = 1}^K \zeta_k Y^{(2)}_{k} \left(  \int_{0}^{t} \left( \lambda_k(\hat{X}_{\theta+h}(s),\theta+h)  - \lambda^{min}_k (\hat{X}_{\theta}(s),\hat{X}_{\theta+h}(s), \theta,h)\right) ds\right), \notag
\end{align}
where $\{ Y_k, Y^{(1)}_k,  Y^{(2)}_k : k =1 ,2,\dots,K \}$ is a family of independent unit rate Poisson processes. The $Y_k$'s here is the same as those in \eqref{tcrep2}. 
These representations define a coupling that was used in \cite{DA} to construct efficient finite difference estimators for the parameter sensitivity.
Note that $\hat{X}_{\theta}$ and $\hat{X}_{\theta+h}$ are Markov processes with generators $\mathbb{A}_{\theta}$ and $\mathbb{A}_{\theta+h}$ respectively. They both start with the same state $x_0$.
Since $\lambda^{min}_k(x,x,\theta,h) \to \lambda_k(x,\theta)$ as $h \to 0$, we must have that $\hat{X}_\theta$ and $\hat{X}_{\theta + h}$ converges almost surely to $X_\theta$ as $h \to 0$. This convergence is in the Skorohod topology on the space $\N^d_0$. For details on this topology see Chapter 3 in \cite{EK}.

By definition 
\begin{align}
\label{sens_defn2}
S_\theta (f,T) = \lim_{h \to 0 } \frac{\E \left( f(\hat{X}_{\theta+h}(T)) \right) - \E \left( f(\hat{X}_\theta(T)) \right)}{h}.
\end{align}
Recall the definition of the generator $\mathbb{A}_{\theta}$ from \eqref{genctmc2}. Part (B) of Lemma \ref{lemma1} gives us the following relationships
 \begin{align}
 \E\left( f(\hat{X}_\theta(T)) \right) & = f(x_0) + \sum_{k = 1}^K \E \left(  \int_{0}^T  \lambda_k (\hat{X}_\theta(t), \theta ) \Delta_{\zeta_k} f(\hat{X}_\theta(t))dt \right) \label{dynkin1}   \textrm{ and } \\
\E\left( f(\hat{X}_{\theta + h }(T)) \right) & = f(x_0) + \sum_{k = 1}^K \E \left(  \int_{0}^T  \lambda_k(\hat{X}_{\theta+h}(t), \theta + h ) \Delta_{\zeta_k} f(\hat{X}_{\theta+h}(t)) dt \right). \label{dynkin2}
\end{align}•
Part (A) of Condition \ref{cond:propensityfunctions} allows us to write the Taylor expansion below
\begin{align*}
\lambda_k(x,\theta+h) = \lambda_k(x,\theta) + h \frac{ \partial \lambda_k (x,\theta)}{\partial \theta} + \frac{h^2}{2} \frac{ \partial^2 \lambda_k (x,\xi)}{\partial \theta^2},
\end{align*}
where $\xi \in (\theta, \theta + h)$. Substituting this expansion in \eqref{dynkin2} we get
\begin{align*}
 \E\left( f(\hat{X}_{\theta + h }(T)) \right)  & =  f(x_0) + \E \left(  \int_{0}^T \mathbb{A}_{\theta} f(\hat{X}_{\theta+h}(t)) dt \right) \\&
 + h \sum_{k = 1}^K \E \left(  \int_{0}^T  \frac{ \partial \lambda_k(\hat{X}_{\theta + h}(t), \theta)}{\partial \theta} \Delta_{\zeta_k} f(\hat{X}_{\theta+h}(t)) dt \right) \\
& +  \frac{ h^2 }{2} \sum_{k = 1}^K \E \left(  \int_{0}^T   \frac{ \partial^2 \lambda_k ( \hat{X}_{\theta+h} (t) ,\xi)}{\partial \theta^2}\Delta_{\zeta_k} f(\hat{X}_{\theta+h}(t)) dt \right). 
\end{align*}
Using \eqref{dynkin1} and \eqref{sens_defn2} we obtain
\begin{align*}
S_\theta (f,T) & =\lim_{h \to 0} \frac{1}{h}  \E \left(  \int_{0}^T \left( \mathbb{A}_{\theta} f(\hat{X}_{\theta+h}(t)) - \mathbb{A}_{\theta}f(\hat{X}_{\theta}(t) )  \right)  dt \right) 
\\&+ \lim_{h \to 0} \sum_{k = 1}^K \E \left(  \int_{0}^T  \frac{ \partial \lambda_k(\hat{X}_{\theta + h}(t), \theta)}{\partial \theta} \Delta_{\zeta_k} f(\hat{X}_{\theta+h}(t)) dt \right) \\
& + \lim_{h \to 0}   \frac{h}{2} \sum_{k = 1}^K \E \left(  \int_{0}^T   \frac{ \partial^2 \lambda_k ( \hat{X}_{\theta+h} (t) ,\xi)}{\partial \theta^2}  \Delta_{\zeta_k} f(\hat{X}_{\theta+h}(t)) dt \right).
\end{align*}
Note that part (B) of Condition \ref{cond:propensityfunctions} says that $\lambda_k$ and $\partial \lambda_k/ \partial \theta$ satisfy Condition \ref{cond:bddness}. Moreover by part (C) of Condition \ref{cond:propensityfunctions}, for some $\epsilon >0$, $\sup_{ \theta' \in (\theta - \epsilon, \theta+ \epsilon)} | \partial^2 \lambda_k ( \cdot,\theta' )/ \partial \theta^2  |$ also satisfies Condition \ref{cond:bddness}. Since both 
$\hat{X}_\theta$ and $\hat{X}_{\theta+h}$ converge to $X_\theta$ almost surely as $h \to 0$, we can use part (A) of Lemma \ref{lemma1} along with the dominated convergence theorem to conclude that the third limit on the right of the above expression is $0$ and the second limit is equal to 
\begin{align*}
\sum_{k = 1}^K \E \left(  \int_{0}^T  \frac{ \partial \lambda_k( X_{\theta}(t), \theta)}{\partial \theta} \Delta_{\zeta_k} f(X_{\theta}(t))  dt \right).
\end{align*}
Therefore
 \begin{align}
\label{sens1:1}
S_\theta (f,T)  = \sum_{k = 1}^K \E \left(  \int_{0}^T  \frac{ \partial \lambda_k( X_{\theta }(t), \theta)}{\partial \theta} \Delta_{\zeta_k} f(X_{\theta}(t)) dt \right) + \rho(\theta) 
\end{align}
 where
 \begin{align}
 \label{defn:rhotheta}
\rho(\theta) = \lim_{h \to 0} \frac{1}{h} \E \left(  \int_{0}^T \left( \mathbb{A}_{\theta} f(\hat{X}_{\theta+h}(t)) - \mathbb{A}_{\theta}f(\hat{X}_{\theta}(t) )  \right)  dt \right). 
\end{align}

For each $k = 1,\dots,K$ let
\begin{align*}
\tau^h_k =  & \inf \left\{ t \geq 0 :  Y^{(1)}_k \left(   \int_{0}^{t} \left( \lambda_k ( \hat{X}_{\theta}(s),\theta)  - \lambda^{min}_k (\hat{X}_{\theta}(s),\hat{X}_{\theta+h}(s), \theta,h)\right) ds\  \right)  \right. 
\\ & \left. + Y^{(2)}_k \left( \int_0^t \left( \lambda_k(\hat{X}_{\theta+h}(s),\theta+h)  - \lambda^{min}_k (\hat{X}_{\theta}(s),\hat{X}_{\theta+h}(s), \theta,h) \right) ds\right)  \geq 1   \right\}.
\end{align*}
Then the first time $\hat{X}_\theta$ and $\hat{X}_{\theta+h}$ have a different state is given by
\begin{align}
\label{defntauh}
\tau^h =  \min_{k} \tau^h_k.
\end{align}
Almost sure convergence of $\hat{X}_\theta$ and $\hat{X}_{\theta+h}$ to $X_\theta$ implies that $\tau^h \to \infty$ as $h \to 0$.
We can write
\begin{align}
\label{rhothetaexpansion}
\rho(\theta) & =  \lim_{h \to 0}  \frac{1}{h} \E \left(  \int_{T \wedge \tau^h}^T \left( \mathbb{A}_{\theta} f(\hat{X}_{\theta+h}(t)) - \mathbb{A}_{\theta}f(\hat{X}_{\theta}(t) )  \right)  dt \right)    = \lim_{h \to 0} \sum_{k = 1}^K  \sum_{i = 0}^{\infty} \frac{1}{h} \rho^h_{ik}(\theta),
\end{align}
 where
 \begin{align}
\label{defn:rhohiktheta}
\rho^h_{ik}(\theta) = \E \left(   \ind_{ \{  \sigma^h_i < \tau^h < \sigma^h_{i+1} , \tau^h = \tau^h_k  \}}  \int_{T \wedge \tau^h}^T \left( \mathbb{A}_{\theta} f(\hat{X}_{\theta+h}(t)) - \mathbb{A}_{\theta}f(\hat{X}_{\theta}(t) )  \right)  dt \right)
\end{align}• 
 and $\sigma^h_i$ is the $i$-th jump time of the process $\hat{X}_\theta$ (with $\sigma^h_0 = 0$). Let $\{\mathcal{F}^h_t \}$ be the filtration generated by the processes $\hat{X}_\theta$ and $\hat{X}_{\theta+ h}$. Note that since Poisson processes are strongly Markov, the processes $\hat{X}_\theta$ and $\hat{X}_{\theta+h}$ are also strongly Markov.
 Conditioning with respect to $\mathcal{F}_{ T \wedge \tau^h}$ and using part (B) of Lemma \ref{lemma1} along with the strong Markov property we get
 \begin{align*}
\rho^h_{ik}(\theta) &= \E \left(   \ind_{ \{  \sigma^h_i < \tau^h < \sigma^h_{i+1} , \tau^h = \tau^h_k  \}}  \E \left( \int_{T \wedge \tau^h}^T \left( \mathbb{A}_{\theta} f(\hat{X}_{\theta+h}(t)) - \mathbb{A}_{\theta}f(\hat{X}_{\theta}(t) )  \right)  dt   \middle \vert \mathcal{F}_{\tau^h \wedge T} \right) \right) \\
& = \E \left(   \ind_{ \{  \sigma^h_i < \tau^h < \sigma^h_{i+1} , \tau^h = \tau^h_k  \}}  \E \left( D_f( \hat{X}_{\theta+h}, T \wedge \tau^h, T ) - D_f( \hat{X}_{\theta}, T \wedge \tau^h, T ) 
\middle \vert \mathcal{F}_{\tau^h \wedge T} \right) \right),
\end{align*}
where $D_f(X,s,t) = f(X(t)) - f(X(s)).$
Let $\gamma^h_i = (\tau^h - \sigma^h_{i}) \wedge (\sigma^h_{i+1} - \sigma^h_{i})$. Given $\tau^h > \sigma^h_i$ and $\hat{X}_\theta(\sigma^h_i) = x$, $\gamma^h_i$ is exponentially distributed with rate 
\begin{align*}
r(x,\theta,h) = \lambda_0(x,\theta) + \lambda_0(x, \theta+h) - \lambda^{min}_0(x,x,\theta,h)
\end{align*}•
and the probability of the event $E^h_{ik} = \{  \tau^h < \sigma^h_{i+1} , \tau^h = \tau^h_k  \} = \{  \tau^h = \sigma^h_i + \gamma^h_i , \tau^h = \tau^h_k  \} $ is given by
\begin{align*}
p^h_{ik}(x) = \frac{\lambda_k(x,\theta) + \lambda_k(x, \theta+h) - 2  \lambda^{min}_k(x,x,\theta,h)}{r(x,\theta,h)} 
\end{align*}
Using Taylor's theorem we can write
\begin{align}
\label{proof_prob2}
 p^h_{ik} (x) = \frac{1}{r(x,\theta,h)} \left( \left|  \frac{\partial \lambda_k(x, \theta)}{\partial \theta }   \right| h +  l^\theta_{ik}(h,x)h \right),
\end{align}
where $l^\theta_{ik}(h,x) \to 0$ as $h \to 0$. Given $\sigma^h_i < \tau^h <T $ and $\hat{X}_\theta(\sigma^h_i) = x$, on the event $E^h_{ik}$ we have
\begin{align}
\label{icforx1x2_1}
(\hat{X}_{\theta}(\tau^h ),\hat{X}_{\theta+h}(\tau^h ) ) = \left\{
\begin{array}{cc}
 ( x, x + \zeta_k ) & \textrm{ if }  \lambda_k(x,\theta+h) > \lambda_k(x,\theta) \\
 ( x + \zeta_k , x ) & \textrm{ if }  \lambda_k(x,\theta+h) < \lambda_k(x,\theta) .\\
\end{array}\right.
\end{align}
Recall the definition of $\Psi_\theta$ from \eqref{defpsixtheta}. Let 
\begin{align*}
\hat{R}^h_{\theta}(x,f,t,k) & =  \int_{0}^{t} \left( \Psi_{\theta + h}(x  + \zeta_k ,f, t- s)  -  \Psi_{\theta} (x,f,t - s) - \Delta_{\zeta_k}  f( x) \right) e^{ -r(x ,\theta,h) s } d s \\
& =  \int_{0}^{t} \left( \Psi_{\theta + h}(x  + \zeta_k ,f, s)  -  \Psi_{\theta} (x,f,s) - \Delta_{\zeta_k}  f( x) \right) e^{ -r(x ,\theta,h) (t - s) } d s. 
\end{align*}
Suppose that for some $x \in \N^d_0$ we have $(\partial \lambda_k(x, \theta)/ \partial \theta )  > 0$. This implies that $\lambda_k(x,\theta+h) > \lambda_k(x,\theta)$ for $h$ close to $0$. 
Since $\gamma^h_i$ is exponentially distributed with rate $r(x,\theta,h)$ we obtain
\begin{align*}
& \frac{1}{h} \E \left(   \ind_{ E^h_{ik}} \E \left( D_f( \hat{X}_{\theta+h}, T \wedge \tau^h, T ) - D_f( \hat{X}_{\theta}, T \wedge \tau^h, T )  \middle \vert \mathcal{F}_{\tau^h \wedge T} \right)  
\middle \vert \hat{X}_\theta(\sigma^h_i) = x\right) \\
& = \frac{1}{h} \E \left(   p^h_{ik}(x)   r(x,\theta,h)\hat{R}^h_{\theta}(x,f, T - \sigma^h_i \wedge T  ,k) \right).
\end{align*}
Using \eqref{proof_prob2} one can see that
\begin{align*}
& \lim_{h \to 0}\frac{1}{h} \E \left(   \ind_{ E^h_{ik}} \E \left( D_f( \hat{X}_{\theta+h}, T \wedge \tau^h, T ) - D_f( \hat{X}_{\theta}, T \wedge \tau^h, T )  \middle \vert \mathcal{F}_{\tau^h \wedge T} \right)  \middle \vert \hat{X}_\theta(\sigma^h_i) = x\right)  \notag \\
& =\lim_{h \to 0} \E \left(    \frac{\partial \lambda_k(x, \theta)}{\partial \theta }     \hat{R}^h_{\theta}(x,f, T - \sigma^h_i \wedge T  ,k)  \right). 
\end{align*}
The above relation will also hold when $(\partial \lambda_k(x, \theta)/ \partial \theta )  < 0$. We can now conclude that
\begin{align*}
 \lim_{h \to 0} \frac{ \rho^h_{ik}(\theta)  }{h} = & \lim_{h \to 0} \E \left[  \ind_{\{ \tau^h > \sigma^h_i \} } \frac{\partial \lambda_k(\hat{X}_\theta(\sigma^h_i) , \theta)}{\partial \theta }  \hat{R}^h_{\theta}(\hat{X}_\theta(\sigma^h_i)   ,f,T - \sigma^h_i \wedge T ,k) \right].
\end{align*}
As $h \to 0$, the processes $\hat{X}_\theta$ and $\hat{X}_{\theta+h}$ converge almost surely to $X_\theta$. This implies that as $h \to 0$,
$\sigma^h_i \to \sigma_i$ and $r(\hat{X}_\theta(\sigma^h_i) ,\theta,h) \to \lambda_0(X_\theta(\sigma_i), \theta )$ almost surely, where $\sigma_i$ is the $i$-th jump time of the process $X_\theta$.
Also recall that $\tau^h \to \infty$ as $h \to 0$. Therefore using the continuous mapping theorem, one can see that
\begin{align}
\label{thm:senslimit1}
\lim_{h \to 0} \frac{\rho^h_{ik}(\theta)}{h} &= \E \left[ \frac{\partial \lambda_k(X_\theta(\sigma_i) , \theta)}{\partial \theta }   R_\theta( X_\theta (\sigma_i) ,f, T - \sigma_i  \wedge T ,k)
\right] ,
\end{align}
where the function $R_\theta$ is given by \eqref{defrktheta}.
Using \eqref{sens1:1},  \eqref{rhothetaexpansion} and \eqref{thm:senslimit1} we obtain
\begin{align*}
S_{\theta}(f,T) =  \sum_{k = 1}^K & \E \left(  \int_{0}^T  \frac{ \partial \lambda_k( X_{\theta }(t), \theta)}{\partial \theta} \Delta_{\zeta_k} f(X_{\theta}(t)) dt \right. \\& \left. +\sum_{i = 0 : \sigma_i < T}^{\infty} \frac{\partial \lambda_k(X_\theta(\sigma_i) , \theta)}{\partial \theta }   R_\theta( X_\theta (\sigma_i) ,f, T - \sigma_i ,k) \right). 
\end{align*}
This completes the proof of Theorem \ref{thm:sens} .
\end{proof}

\section{Conclusions and Future Work}
In this paper we present a new approach for obtaining unbiased estimates for the parameter sensitivities in a stochastic chemical reaction network. 
Our approach depends on a sensitivity formula that is derived using the random time change representation of Kurtz and the coupling introduced in \cite{DA}. The derived formula involves several quantities that cannot be directly computed from the trajectories of the reaction dynamics and hence they need to be estimated through simulations. We present a procedure called the \emph{Auxiliary Path Algorithm} (APA) that efficiently estimates all the required quantities using a fixed number of auxiliary paths. 
It relies on the fact that a Markov process representing chemical kinetics is such that its different paths visit the same states again and again. 
With APA we can generate samples that can be used for estimating the parameter sensitivity. 

Using the birth-death model (see Example \ref{example_bd2}) and the gene expression network (see Example \ref{example_ge}), we compare APA to the Girsanov method \cite{Gir}, which is the only other method known to produce unbiased estimates for the parameter sensitivity. Our results show that APA is considerably faster than the Girsanov method when the sensitivity is estimated with respect to a reaction rate constant which is \emph{small} in size. Such a rate constant would correspond to a reaction which is \emph{slow} in the reference time-scale of the system. Many biological networks have such \emph{slow} reactions and this makes our method useful for sensitivity analysis.
The main drawbacks of APA is that it is hard to implement and requires memory of a size proportional to the number of jumps in a typical path. 

Our method assumes that we can efficiently simulate the paths of the reaction network using Gillespie's \emph{Stochastic Simulation Algorithm} \cite{GP}. 
However this is not true when the chemical system has reactions that are taking place at vastly different time-scales (see \cite{ssSSA,weinan1,weinan2}). In such situations, our method will perform poorly but so will all the other methods for estimating the parameter sensitivity. Recently Kang and Kurtz \cite{HWKang} have developed a formal framework for reducing stochastic reaction models by exploiting the time-scale separation between the various reactions. These reduced models \emph{approximately} capture the dynamics of the original model and they can be efficiently simulated using a variant of the standard \emph{Stochastic Simulation Algorithm} (see \cite{ssSSA,weinan1,weinan2}). It is natural to ask if one can just simulate the reduced model and  \emph{approximately} estimate the parameter sensitivity for the original model. In an upcoming paper we will show that this can indeed be done under some conditions. That result will follow from ideas that are similar to the ones used in the proof of Theorem \ref{thm:sens}.

In some applications, one is interested in estimating the second order derivatives (the Hessian) with respect to the model parameters. One such application is optimization over the parameter space, where the Hessian is required to implement the Newton-Raphson scheme. Our approach in this paper can also yield exact expressions for the second order derivatives. One can then use such expressions to build unbiased estimators for the Hessian. We hope to do that in a future paper.

\bibliographystyle{abbrv}
\bibliography{references}

\end{document}